\documentclass[a4paper, 11pt]{amsart}

\usepackage{amssymb,amscd,amsmath}
\usepackage[mathcal]{eucal}
\usepackage{array,float}

\usepackage{xcolor}
\usepackage{array}
\usepackage{xy}
\input xy
\xyoption{all}
\usepackage{pdflscape}
\usepackage{hyperref}
\usepackage[left=3cm, right=3cm]{geometry}
\usepackage{enumerate}
\numberwithin{equation}{section}
\numberwithin{equation}{subsection}
\usepackage{amsmath}

\newtheorem{thm}{Theorem}[section]
\newtheorem{maintheorem}{Theorem}

\newtheorem{corollary}{Corollary}[section]
\newtheorem{maincorollary}{Corollary}

\newtheorem{defn}[thm]{Definition}

\newtheorem{definition}[thm]{Definition}
\newtheorem*{remark*}{Remark}
\newtheorem{remark}[thm]{Remark}

\newtheorem{example}[thm]{Example}

\newtheorem{question}[thm]{Question}

\newcommand{\Hom}{{\mathrm{Hom}}}

\newcommand{\tra}{{\mathrm{tra}}}

\newcommand{\M}{\mathrm{M}}

\newcommand{\bZ}{{\mathbb Z}}
\newcommand{\bC}{{\mathbb C}}

\title[On the Twisted group ring isomorphism  problem  for isoclinic groups]{ On the Twisted group ring isomorphism  problem  for isoclinic groups}

\author{Sumana Hatui$^{a,b}$}
\author{Sahanawaj Sabnam$^{a,b}$}

\address{$^{a}$School of Mathematical Sciences,
National Institute of Science Education and Research,
Bhubaneswar 752050, Odisha, India.}

\address{$^{b}$Homi Bhabha National Institute,
Training School Complex,
Anushakti Nagar,
Mumbai 400094, India.}

\email{sumanahatui@niser.ac.in}
\email{sahanawaj.sabnam@niser.ac.in}

\begin{document}



 \begin{abstract}
In this article, we consider a general version of the classical group ring isomorphism problem, called the twisted group ring isomorphism problem (TGRIP), which determines an isomorphism between the twisted complex group algebras of finite groups. Although for isoclinic groups $G$ and $H$ their complex group algebras are isomorphic, i.e,  $\mathbb C G \cong \mathbb C H$, their twisted complex group algebras need not be isomorphic. 
Continuing this line of investigation, we establish sufficient conditions under which two isoclinic groups have isomorphic twisted complex group algebras, thereby providing solutions to (TGRIP) for isoclinic groups. As applications, we study (TGRIP) for special $p$-groups of rank $2$, unicentral groups, and nilpotent groups of class $2$ with elementary abelian Schur multipliers, all considered up to isoclinism. We also present several examples illustrating the main results, including a complete solution of (TGRIP) for special $p$-groups of rank $2$ of order $p^6(p \ge 3)$.
\end{abstract}

   \maketitle 

 \section{Introduction and Main results}
Let $G$ be a finite group and $R$ a commutative ring with identity. The group ring $RG$ serves as a natural bridge between group theory and representation theory, encoding substantial structural information about $G$. A central question in the classical group ring isomorphism problem is to characterize finite groups $G$ and $H$ satisfying $RG \cong RH$. The answer to this problem depends strongly on the choice of the coefficient ring $R$, and the problem has attracted considerable attention over the past several decades.

As a refinement of this classical problem, twisted group rings were introduced. Let $R^{\times}$ denote the group of units of $R$. The set of all $2$-cocycles of a finite group $G$ with values in $R^{\times}$ is denoted by $Z^{2}(G, R^{\times})$, and the corresponding second cohomology group by $\mathrm H^{2}(G, R^{\times})$. For $\alpha \in Z^{2}(G, R^{\times})$, its cohomology class is denoted by $[\alpha] \in \mathrm H^{2}(G, R^{\times})$.
For $\alpha$, the corresponding twisted group ring $R^{\alpha}G$ is defined as the free $R$-module with basis $\{ \bar{g} \mid g \in G \}$, where the multiplication is defined by
\[
\bar{g}\,\bar{h} = \alpha(g,h)\overline{gh}, \quad \text{for all } g,h \in G.
\]
Motivated by the classical group ring isomorphism problem, Margolis and Schnabel introduced in \cite{Margolis}, its projective analogue, known as the twisted group ring isomorphism problem (TGRIP), which is defined as follows.
\begin{definition}\label{TGRIP def}
Let $G$ and $H$ be finite groups. We say that $G \sim_R H$ if there exists an isomorphism
$\psi : \mathrm H^{2}(G, R^{\times}) \to \mathrm H^{2}(H, R^{\times})$
such that
$R^{\alpha}G \cong R^{\psi(\alpha)}H$
for every $[\alpha] \in \mathrm H^{2}(G, R^{\times})$.
\end{definition}
The twisted group ring isomorphism problem (TGRIP) asks to determine the equivalence classes of groups under the equivalence relation $\sim_R$. Following \cite[p.\ 2]{Pooja}, these equivalence classes are called the $R$-twist isomorphism classes. Throughout this paper, $R=\mathbb C$.

Although the (TGRIP) is a natural extension of the classical group ring isomorphism problem, it is substantially more intricate because twisted group rings depend not only on the underlying group but also on its second cohomology group. Significant progress has recently been made on this problem. Margolis and Schnabel \cite{Margolis} proved that every finite abelian group forms a singleton $\mathbb C$-twist isomorphism class. They further determined the $\mathbb C$-twist isomorphism classes of groups of order $p^{4}$ and $p^{2}q^{2}$, where $p$ and $q$ are distinct primes. Subsequently, in \cite{Pooja}, the classification of $\mathbb C$-twist isomorphism classes for non-abelian $p$-groups was carried out by fixing their generalized corank. More recently, in \cite{TGRIP}, we provide a criterion to reduce the (TGRIP) for certain groups to the corresponding problem for suitable quotient groups. For the direct product of groups, we propose several conditions to determine the (TGRIP) of the group by considering its components, and we also give a criterion for studying the (TGRIP) of the central product of groups. As an application, we provide an answer for extra-special $p$-groups, and finally, the $\mathbb {C} $-twist isomorphism classes of groups of order $p^{5}$, for $p \ge 5 $, were completely determined.

The present paper is motivated by understanding the (TGRIP) for isoclinic groups. Since many invariants arising in ordinary representation theory are preserved under isoclinism, it is natural to ask whether an analogous phenomenon holds for twisted group rings and projective representations. 

We refer the reader to Section~\ref{notation} for the notation and terminology used throughout the paper. 
First, we recall the definition of isoclinism from \cite[Theorem 1.2]{Tappe}, which is equivalent to the classical definition.
 \begin{definition}\label{iso2}
 Let $G_1$ and $G_2$ be two groups and $A\subseteq Z(G_1)$, $B\subseteq Z(G_2)$ be central subgroups.
We say that $G_1$ and $G_2$ are isoclinic via the central subgroups $(A, B)$
if there are isomorphisms $\phi: G_1/A \to G_2/B$ and $\delta: G_1' \to G_2'$ such that
\[
\delta([g_1,g_2]) = [h_1,h_2]
\]
for all $g_1, g_2 \in G_1$ and $h_1 \in \phi(g_1A)$, $h_2 \in \phi(g_2A)$. 

Unless otherwise stated, when we say that two groups $G_1$ and $G_2$ are isoclinic, we mean that $A = Z(G_1)$ and $B = Z(G_2)$. 
When we want to specify the corresponding isomorphisms, we also add that $G_1$ and $G_2$ are isoclinic via  $(\phi,\delta)$.
\end{definition}
 It is well known that if $G$ and $H$ are two isoclinic groups of the same order, then
$\mathbb{C}G \cong \mathbb{C}H$
(see \cite[p. 136]{hall}, \cite[Theorem 3.2]{Tappe}). 
Consequently, they have the same irreducible complex character degrees with identical multiplicities. From the perspective of projective representations, one might expect that two isoclinic groups with isomorphic Schur multipliers should also belong to the same $\mathbb {C}$-twist isomorphism class. Surprisingly, the following example shows that this is not true in general.
\begin{example}[{\cite[Theorem 4.3]{TGRIP}}]
The groups $\Phi_6(2111)a$ and $\Phi_6(221)d_0$ of order $p^5$ belong to the same isoclinism family and satisfy
$\M(\Phi_6(2111)a) \cong \M(\Phi_6(221)d_0).$
Nevertheless,
$\Phi_6(2111)a \nsim_{\mathbb{C}} \Phi_6(221)d_0.$
\end{example}
This example naturally led to the following question.
\begin{question}\label{question}
Let $G_1$ and $G_2$ be isoclinic groups of the same order satisfying
$
\M(G_1)\cong\M(G_2).
$
Under what  conditions does
$
G_1\sim_{\mathbb C}G_2
$
hold?
\end{question}
In this article, our main goal is to answer Question~\ref{question}  by providing various sufficient conditions on isoclinic groups. Our first main result  addresses the non-capable isoclinic groups. 
\begin{maintheorem}\label{TGRIP isoclinic}
Let $G_1$ and $G_2$ be two non-capable groups of the same order. Suppose $
A_i \subseteq Z^*(G_i)\cap G_i'
$
for $i=1,2$ such that $G_1$ and $G_2$ are isoclinic via 
the central subgroups $(A_1, A_2)$. Then
$
G_1 \sim_{\mathbb{C}} G_2.
$
\end{maintheorem} 
 We now recall the notion of a representation group of $G$. 
\begin{defn}[\cite{karpischur}, p.16]\label{rep group} 
A group $\tilde{G}$ is called a \emph{representation group} (or \emph{covering group}) of a group $G$ if there exists a central extension
\[
1 \to A \to \tilde{G} \to G \to 1
\]
such that
$A \subseteq Z(\tilde{G}) \cap \tilde{G}'
~ \text{and} ~
A \cong \M (G).
$
\end{defn}
Schur (\cite[Theorem 2.1.4]{karpischur}) proved that every finite group admits a representation group. Then 
using \cite[Chapter 3, Corollary 3.4, Theorem 3.7]{karpiprojective}, by Artin-Wedderburn  decomposition   we have 
\[
\mathbb C\tilde G
\cong
\bigoplus_{[\alpha]\in\M(G)}
\mathbb C^{\alpha}G.
\]
Indeed, if $\tilde G_i$ is a representation group of $G_i$ for $i=1,2$, then a necessary condition for
$
G_1\sim_{\mathbb C}G_2
$
is that
$
\mathbb C\tilde G_1
\cong
\mathbb C\tilde G_2.
$
However, Margolis and Schnabel showed  \cite[Theorem 1.6(f)]{Margolis} that this condition  is not sufficient.  
 Motivated by this fact, we establish the following result.
\begin{maintheorem}\label{representation related}
Let $G_1$ and $G_2$ be groups of the same order with isomorphic Schur multipliers. If their representation groups are isoclinic, then
$
G_1 \sim_{\mathbb{C}} G_2.
$
\end{maintheorem}
The examples given in Remark~\ref{non-isoclinic rep-group} show that the hypothesis on the representation groups cannot, in general, be omitted. Indeed, when the representation groups are not isoclinic, the corresponding groups may or may not satisfy the (TGRIP). Thus, Theorem~\ref{representation related} is best possible in this sense.

As an application of Theorem~\ref{representation related}, we obtain the following corollary, which provides a  several new families $p$-groups of nilpotency class $2$ satisfying the (TGRIP). 
 Let $G$ be a group of nilpotency class $2$ that is minimally generated by a set
 $X=\{g_1,g_2,\ldots,g_d\}.$
Define
$$W_{G}^{X}:=\{(g_i,g_j)\in X\times X \mid [g_i,g_j]=1,\ i<j\}.$$

\begin{maincorollary}\label{class 2}
Let $G_1$ and $G_2$ be two isoclinic $p$-groups of the same order satisfying $d(G_1)=d(G_2),$ and $G_1$ admits a minimal generating set
$$X=\{x_1,\ldots,x_m,y_1,\ldots,y_r\},$$
$$\text{ where } ~~~~x_i\notin Z(G_1)\quad (1\le i\le m), \qquad
y_j\in Z(G_1)\quad (1\le j\le r).$$
Suppose that the following conditions hold:
\begin{enumerate}[(i)]
    \item $G_1$ and $G_2$ are of nilpotency class $2$;
    
    \item$\M (G_1)\cong \M (G_2)\cong (\mathbb{Z}/p\mathbb{Z})^k;$
    
    \item $|W_{G_1}^{X}| \ge k.$
\end{enumerate}
Then $G_1 \sim_{\mathbb{C}} G_2.$
\end{maincorollary} 
Proofs of the above results are provided in Section~\ref{Proof} and we also present several examples illustrating applications of these results in Section \ref{Application_1}.

We next investigate the (TGRIP) for isoclinic special $p$-groups. Recall that  a finite $p$-group is called a special $p$-group of rank $k$ if $G' = Z(G)$ is an elementary abelian group of order $p^k$ and $G/G'$ is also elementary abelian. In the particular case $k=1$, these are precisely the extra-special $p$-groups. The (TGRIP) for extra-special $p$-groups was completely resolved in \cite[Theorem 1.7]{TGRIP}. Our next result yields further solutions to  (TGRIP) for isoclinic special $p$-groups.
 \begin{maintheorem} \label{TGRIP special p group}
Let $G_1$ and $G_2$ be two isoclinic  non-capable special $p$-groups of rank $k$. Suppose that one of the following conditions holds.
\begin{enumerate}
\item $Z^*(G_i)=Z(G_i);$
\item $k=2$ and $\M (G_1) \cong \M (G_2)$. 
\end{enumerate}
Then
$
G_1 \sim_{\mathbb{C}} G_2.
$
\end{maintheorem}
\begin{remark}
For capable special $p$-groups with $k=2$, it follows from \cite[Proposition~3]{heineken} that
 $p^5 \leq |G| \leq p^7$. Consequently, using the presentations given in \cite[Theorem~4.1, Theorem~5.1]{mazur}, one can verify the corresponding cases computationally.
\end{remark}
The proof of the above result is given in Section~\ref{Special p-groups}. As an application, we obtain a solution to the (TGRIP) for all special $p$-groups of rank $2$ of order $p^6$ in Section \ref{Application_2}.

Let $P$ denote the collection of pairs $(G_1, G_2)$ of non-capable special $p$-groups of the same rank satisfying the hypotheses of Theorem~\ref{TGRIP special p group}. Combining Theorem \ref{TGRIP special p group} with \cite[Corollary~1.5]{TGRIP}, we obtain the following result, which provides new solutions to (TGRIP).
\begin{maincorollary}
The following hold.
\begin{enumerate}
\item If $(G_1,G_2)\in P$, then for any finite abelian group $A$, we have
$
G_1 \times A \sim_{\mathbb{C}} G_2 \times A.
$

\item If $(G_1,G_2),(H_1,H_2)\in P$, then
$
G_1 \times H_1 \sim_{\mathbb{C}} G_2 \times H_2.
$
\end{enumerate}
\end{maincorollary}  
 
  \subsection{Notations}\label{notation} 
For a finite group $G$, let $d(G)$ denote the minimum size of a generating set of $G$, and let $G^{p}$ denote the subgroup generated by all $p$-th powers of elements of $G$. 
The center and the commutator subgroup of a group $G$ are denoted by $Z(G)$ and $G'$, respectively, and the commutator of two elements $g,h\in G$ is defined by
$[g,h]=g^{-1}h^{-1}gh.$ The set of all inequivalent irreducible complex ordinary representations of $G$ is denoted by $\mathrm{Irr}(G)$. For a normal subgroup $Z$  of $G$ and for $\chi \in \mathrm{Irr}(Z)$, we denote $\mathrm{Irr}(G\mid\chi)=\{\rho\in \mathrm{Irr}(G)\mid \langle \rho|_Z,\chi\rangle\neq 0\}$, the set of all inequivalent irreducible ordinary representations of $G$ lying above $\chi$. A group $G$ is called capable if there exists a group $H$ such that
$G\cong H/Z(H),$ and non-capable otherwise. The epicenter of $G$, denoted by $Z^*(G)$, is the smallest central subgroup of $G$ such that $G/Z^*(G)$ is capable. A group $G$, satisfying
  $Z^*(G)=Z(G)$ is called unicentral. The cohomology group $\mathrm H^{2}(G,\mathbb C^{\times})$ is the Schur multiplier of $G$, which we denote by $\M(G)$. Whenever presentations of $p$-groups are used, we follow the convention of \cite{james} and omit relations of the form $[g,h]=1$ among generators. We also adopt the notation and presentations from \cite{james} throughout the paper without further reference.

 
\section{Preliminaries}
In this section we recall the results that will be used to prove our main results.

 \begin{thm}[ \cite{karpischur}, Theorem 2.2.10]\label{Schur direct product} 
 	For two groups $G_1$ and $G_2$, 
 	$$\M(G_1\times G_2)\cong \M(G_1) \times \M(G_2)\times \Hom(G_1/G_1' \otimes_\mathbb Z G_2/G_2', \mathbb C^\times).$$
 \end{thm}
Let $G$ be a finite group of nilpotency class $2$ with $G'$ and $G/G'$ being elementary abelian. Now we recall the facts from \cite{evens}, which are useful to compute $\M(G)$.
 View $G/G'$ and $G'$ as vector spaces over $\mathbb{F}_p$ and denote them by $V$ and $W$, respectively. Let $X_1$ be the subspace of $V \otimes W$ spanned by all 
 \[g_1G' \otimes[g_2,g_3] + g_2G' \otimes [g_3,g_1]+g_3G' \otimes [g_1,g_2],\]
 for $g_1,g_2,g_3 \in G$. Consider a map $f: V \to W$ defined by $f(gG')=g^p$ for $g \in G$. Let $X_2$ be the subspace spanned by all $gG' \otimes f(gG'),~gG' \in V$. Now consider the space $X:=X_1 +X_2.$
 \begin{thm}[\cite{schurrank2}, Lemma 2.4] \label{epicentre2}
 	Let $Z$ be a central subgroup of a finite group $G$ of nilpotency class $2$. Then $Z \subseteq Z^*(G)$ if and only if $(G/G') \otimes Z $ is contained in $X$.
 \end{thm}
 Let $G$ be a group and $V$ be a vector space over  $\bC$. A mapping  $\rho:G\rightarrow{GL(V)}$ is called a projective representation of $G$ over $\bC$, if there exists a mapping $\alpha:G \times G \rightarrow \mathbb C^\times$ such that 
 \begin{enumerate}[(i)]
 	\item $\rho (x) \rho (y)=\alpha(x,y) \rho(xy)$ for all $x,y \in G$.
 	\item $\rho(1)=1_V.$
 \end{enumerate}
 We say that $\rho$ is an $\alpha$-representation of $G$, and $\mathrm{Irr}^\alpha(G)$ denotes the set of all linearly inequivalent irreducible $\alpha$-representations of $G$.
  We refer the reader to \cite[Chapter 3]{karpiprojective} for definitions and basic facts on projective representations of groups. 
Now, we recall  inflation and transgression maps, which are defined as follows.
 Let $Z$ be a normal subgroup of $G$. For any $\alpha\in Z^2(G/Z,\mathbb C^\times)$,
define $\beta:G\times G\to \mathbb C^\times$ by $\beta(x,y)=\alpha(xZ,yZ)$, $x,y \in G$.
Then $\beta\in Z^2(G,\mathbb C^\times)$ and the homomorphism $\inf:\M(G/Z)\to \M(G)$,
defined by  $\inf([\alpha])=[\beta]$, is called the inflation map.
Now consider a central extension
 $1\to Z\to G^*\xrightarrow{\,f\,}G\to 1$, 
 and let $\mu:G\to G^*$ be a section of $f$. For any $\chi\in\Hom(Z,\mathbb C^\times)$, define
$\alpha_\chi:G\times G\to \mathbb C^\times$
by $\alpha_\chi(x,y) =\chi\left(\mu(x)\mu(y)\mu(xy)^{-1}\right)$.
Then the homomorphism
$\tra:\Hom(Z,\mathbb C^\times)\to \M(G)$,
which maps $\chi$ to $[\alpha_\chi]$, is called the transgression map associated with the given central extension. 
  
 \begin{thm}[\cite{karpischur}, Theorem 2.5.10]\label{exactseq}
 	If $G$ is a finite group with epicenter $Z^*(G)$, then for every subgroup $Z \subseteq Z^*(G)\cap G'$ of $G$,  the following  sequence 
 	\[
 	1  \xrightarrow{}   \mathrm{Hom}(Z, \mathbb C^\times)  \xrightarrow{\mathrm{tra}}  \M(G/Z)   \xrightarrow{\mathrm{inf}} \M(G)  \xrightarrow{}  1 
 	\]
 	is exact.
 \end{thm}
 \begin{thm} \cite[Lemma 4.2]{Pooja} \label{lifting}
 	Let $\tilde{G}_1$ and $\tilde{G}_2$ be representation groups of $G_1$ and $G_2$ respectively, such that they correspond to the central extensions $1 \to Z_i \to \tilde{G}_i \to G_i\to 1$.
Let $\sigma : \mathrm{Hom}(Z_1, \mathbb C^\times) \to \mathrm{Hom}(Z_2, \mathbb C^\times)$ be an isomorphism such that for every $\chi \in  \mathrm{Hom}(Z_1, \mathbb C^\times)$, $$\mathrm{Irr}(\tilde{G}_1 \mid \chi) \longleftrightarrow \mathrm{Irr}(\tilde{G}_2 \mid \sigma(\chi)).$$
 	Then $G_1 \sim_\mathbb C G_2$.
 \end{thm}

 We say an isomorphism ${\bar{\psi}}:\M (G_1/Z_1) \to \M(G_2/Z_2)$ is induced from an isomorphism $\psi:G_1/Z_1 \to G_2/Z_2$ if $\bar{\psi}$ is defined by  $\bar{\psi}([\alpha])=[\beta]$ such that $\beta(\bar{g}_2,\bar{g}_2')=\alpha(\psi^{-1}(\bar{g}_2),\psi^{-1}(\bar{g}_2'))$, for $\bar{g}_2=g_2Z_2,\bar{g}_2'={g}_2'Z_2 \in G_2/Z_2$.
Sometimes we also denote the inflation and transgression maps by $\inf_i$ and $\tra_i$, respectively ($i=1,2$), which will be clear from the context.
 \begin{thm}\label{firstthm} \cite[Corollary 1.3]{TGRIP}
 	 Let $G_1, G_2$ be groups and let $Z_j$ be a central subgroup of $G_j$ for $j=1,2$ such that the sequences 
 	\[
 	1  \xrightarrow{}   \mathrm{Hom}(Z_{j}, \mathbb C^\times)  \xrightarrow{\mathrm{tra_j}}  \M(G_j/Z_{j})   \xrightarrow{\mathrm{inf}_j}  \M(G_j)  \xrightarrow{}  1 
 	\]
 	are exact. Suppose
 	there are isomorphisms $i: Z_1\to Z_2$ and $\psi: G_1/Z_1 \to G_2/Z_2$ such that the following diagram is commutative for the induced isomorphisms $\bar{i}$ and $\bar{\psi}$. 
 		\begin{figure}[H] \label{diagram1}
 		\[
 		\xymatrix{ 
 			1 \ar[r] &  \mathrm{Hom}(Z_{1}, \mathbb C^\times) \ar[d]^{\bar{i}} \ar[r] ^{\mathrm{tra}_{1}} & \M(G_1/Z_{1})  \ar[d]^{\bar{\psi}} \\
 			1 \ar[r] &  \mathrm{Hom}(Z_{2}, \mathbb C^\times) \ar[r] ^{\mathrm{tra}_{2}} & \M(G_2/Z_{2})\\
 		}
 		\]
 	\end{figure}
	Then $G_1 \sim_\mathbb C G_2$.
 \end{thm}

  \section{Isoclinic groups}\label{Proof}
    In this section, we give proofs of Theorem \ref{TGRIP isoclinic}, Theorem \ref{representation related}, and Corollary \ref{class 2}. 
 Suppose $G_1$ and $G_2$ are groups of the same cardinality such that they are isoclinic via the central subgroups $(A, B)$ and via the isomorphisms $(\phi,\delta)$.
 Then it follows from \cite[Lemma 1.5]{Tappe} that,
 \begin{equation} \label{delta map}
 	\delta(A\cap G_1')= B\cap G_2'.
 \end{equation}
 Thus, the map
 \begin{equation} 
 	\tilde{\delta}:  \Hom (A\cap G_1', \bC^\times) \to \Hom(B\cap G_2',\bC^\times) \nonumber
 \end{equation}
 defined by $\tilde{\delta}(\chi)=\chi \circ (\delta^{-1} |_{B \cap {G_2}^{\prime}})$ for $\chi \in \Hom (A\cap G_1', \bC^\times)$, is an isomorphism.
 Now, every $\chi \in \Hom (A\cap G_1', \bC^\times)$ can be extended to a $\chi_1 \in \Hom (A, \bC^\times)$ in $|\frac{A}{A\cap G_1'}|$ many ways. Since $|\frac{A}{A\cap G_1'}|=|\frac{B}{B\cap G_2'}|$, we can define a bijective map
 \begin{equation} \label{eta map}
 	\eta:\mathrm{Hom}(A, \bC^\times) \to \mathrm{Hom}(B, \bC^\times)
 \end{equation}
 such that 
 $\eta(\chi_1)|_{B \cap G_2'}=\tilde{\delta}(\chi_1|_{A \cap G_1'}).$
 \begin{thm}\label{representation tra}
 	Let $G_1$ and $G_2$ be the groups of the same cardinality such that they are isoclinic  via the central subgroups $(A, B)$ and via the isomorphisms $(\phi,\delta)$. Then the following hold.
 	\begin{enumerate}[(i)]
 		\item	The following diagram is commutative for the induced map ${\bar{\phi}}$ of $\phi$.
 		\begin{figure}[H] \label{diagram1}
 			\[\xymatrix{
 				\mathrm{Hom}(A, \bC^\times) \ar[d] ^{\mathrm{\eta}}\ar[r] ^{\mathrm{tra}_1} & \M(G_1/A) \ar[d]^{\mathrm{\bar{\phi}}}\\
 				\mathrm{Hom}(B, \mathbb C^\times) \ar[r] ^{\mathrm{tra}_2} & \M(G_2/B),}
				\label{figure 1}
 			\] \caption{Diagram 1}
 		\end{figure}	
 		
 		
 		%
 		\item For $\chi \in \mathrm{Hom}(A, \bC^\times)$, there is a dimension preserving bijection between the sets 
 		$\mathrm{Irr}(G_1|\chi)$, $ \mathrm{Irr}(G_2|\eta(\chi))
 		$, $ \mathrm{Irr}^{\mathrm{tra_1}(\chi)}(G_1/A)$ and    $\mathrm{Irr}^{\mathrm{tra_2}(\eta(\chi))}(G_2/B)$.
 	\end{enumerate}
 	
 \end{thm}
 
 \begin{proof}
 	(i) Proof follows from \cite[Theorem 1.7]{Tappe}.
 	
 	(ii) Consider the natural exact sequence $1 \rightarrow A \rightarrow G_1 \rightarrow G_1/A \rightarrow1$ and
  $\mu:G_1/A\to G_1$ a  section. For $\chi \in \mathrm{Hom}(A, \bC^\times)$, define a map
 	\[\zeta:\mathrm{Irr}(G_1|\chi) \to \mathrm{Irr}^{\mathrm{tra_1}(\chi)}(G_1/A)\] 
 	by $\zeta(\gamma)= \rho$, where $\rho(\bar{g}):=\gamma(\mu(\bar{g}))$ for $\bar{g}=gA\in G_1/A$. Then $\rho$ is an irreducible $\tra_1(\chi)$-representation of $G_1/A$. Now take another map 
 	\[\xi:\mathrm{Irr}^{\mathrm{tra_1}(\chi)}(G_1/A) \to  \mathrm{Irr}(G_1|\chi)\]
 	defined by $\xi(\rho')=\gamma'$, where $\gamma'(a\mu(\bar{g})):=\chi(a)\rho'(\bar{g})$ for $ a \in A,\bar{g}=gA\in G_1/A$. Since every element $g \in G_1$ can be uniquely written as $a\mu(\bar{g})$, $\gamma'$ is well defined and it is an irreducible ordinary representation of $G_1$ lying above $\chi$. From the definitions of $\zeta$ and $\xi$, we have $\zeta \circ \xi=\xi \circ \zeta =\mathrm{id}$, this yields a bijection between $\mathrm{Irr}(G_1|\chi)$ and $ \mathrm{Irr}^{\mathrm{tra_1}(\chi)}(G_1/A)$.
 	Therefore, to prove the result,  it is enough to give a bijection between
	$\mathrm{Irr}^{\mathrm{tra_1}(\chi)}({G_1}/A)$ and $\mathrm{Irr}^{\mathrm{tra_2}(\eta(\chi))}({G_2}/B)$.		
 	Define a map 
 	\[\Psi:\mathrm{Irr}^{\mathrm{tra_1}(\chi)}(G_1/A) \to \mathrm{Irr}^{\mathrm{tra_2}(\eta(\chi))}({G_2}/B)\]
 	by $\Psi(\rho )=\rho \circ \phi^{-1}$.
 	Then for $\bar{h}=hB,\bar{h}'=h'B \in G_2/B$,
 	\begin{align*}
 		\Psi(\rho)(\bar{h})\Psi(\rho)(\bar{h}') &=	(\rho \circ \phi^{-1})(\bar{h}) (\rho \circ \phi^{-1})(\bar{h}')\\
 		& = \rho (\phi^{-1}(\bar{h}))\rho (\phi^{-1}(\bar{h}'))\\
 		& = \mathrm{tra_1}(\chi)\big(\phi^{-1}(\bar{h}),\phi^{-1}(\bar{h}'\big)\rho (\phi^{-1}(\bar{h})\phi^{-1}(\bar{h}'))\\
 		& = \mathrm{tra_2}(\eta(\chi))(\bar{h},\bar{h}') \Psi(\rho)(\bar{h}\bar{h}').
 	\end{align*}	
 	Therefore, $\Psi(\rho)$ is a $\mathrm{tra_2}(\eta(\chi))$-representation of ${G_2}/B.$  
 	Let $\rho$ and $\rho'$ be two linearly equivalent $\mathrm{tra}_1(\chi')$-representation of $G_1/A$. Then there exist an invertible linear transformation $T$ such that  $\rho(\bar{g})=T\rho'(\bar{g})T^{-1}$. This implies, $\rho (\phi^{-1}(\bar{h}))=T\rho' (\phi^{-1}(\bar{h}))T^{-1}$, for some $\bar{h} \in G_2/B$. Hence $\rho \circ \phi^{-1}$ and $\rho' \circ \phi^{-1}$ are also linearly equivalent $\mathrm{tra_2}(\eta(\chi'))$-representation of $G_2/B.$ So the map is well-defined.
 	Next, define another map 
 	\[\Omega:\mathrm{Irr}^{\mathrm{tra_2}(\eta(\chi))}({G_2}/B) \to \mathrm{Irr}^{\mathrm{tra_1}(\chi)}(G_1/A)\]
 	by $\Omega(\tilde{\rho} )=\tilde{\rho}  \circ \phi$. Then $\Psi \circ \Omega = \Omega \circ \Psi=\mathrm{id}.$ Hence the proof follows.
	
 \end{proof}

 	

 \noindent \textbf{Proof of Theorem \ref{TGRIP isoclinic}.}
 \begin{proof}
 	First note that $\bC G_1 \cong \bC G_2$, since $G_1$ and $G_2$ are isoclinic.
	Now  by Theorem \ref{exactseq} we have the following exact sequence.
 	\[
 	1  \xrightarrow{}   \mathrm{Hom}(A_i, \mathbb C^\times)  \xrightarrow{\mathrm{tra}_i}  \M(G_i/A_i)   \xrightarrow{\mathrm{inf}_i}  \M(G_i)  \xrightarrow{}  1 .
 	\]
If $G_1$ and $G_2$ are isoclinic via the map $(\phi,\delta)$, then by \eqref{delta map}, $\delta(A_1)=A_2$
 	and we have the induced isomorphism
 	$\bar{\delta}$ given by
 	\[\Hom(A_1,\bC^\times)\xrightarrow {\bar{\delta}}\Hom(A_2,\bC^\times)\]
 	\[\chi \mapsto \chi\circ\delta^{-1}.\] 
Since $\phi: G_1/A_1 \to G_2/A_2$ is an isomorphism, the induced map  $\bar{\phi}:\M(G_1/A_1)\to \M(G_2/A_2)$ is an isomorphism. Therefore,  taking $A=A_1$, $B=A_2$, and $\eta=\tilde{\delta}$, it follows from  Theorem \ref{representation tra}$(i)$ that  the Diagram 1 is commutative. Hence by Theorem \ref{firstthm}, $G_1 \sim_\bC G_2$.
\end{proof}
 


 In an isoclinism family, all the groups $G$ satisfying $Z(G) \subseteq G'$ are called stem groups and all the stem groups of a given isoclinism family have the same order \cite[p. 135]{hall}. The following corollary follows immediately from the above theorem, which answer the (TGRIP) for all the unicentral stem groups $G$ belonging to the same isoclinism family.

 \begin{corollary}\label{isoclinic1}
 	Let $G_1$ and $G_2$ be two isoclinic stem groups such that 
 	$Z^*(G_i)=Z(G_i)$. Then $G_1 \sim_\bC  G_2.$ 
 \end{corollary}

 \noindent \textbf{Proof of Theorem \ref{representation related}.}
 \begin{proof}
	Let $\tilde{G}_1$ and $\tilde{G}_2$ be representation groups of $G_1$ and $G_2$ respectively. Suppose $\tilde{G}_1$, $\tilde{G}_2$ are isoclinic via the map $(\phi,\delta)$.
 	We have a central extension 
 	$$1\rightarrow A_i \rightarrow \tilde{G}_i \rightarrow G_i \rightarrow 1$$ such that $A_i \cong \M(G_i)$ and $A_i \subseteq Z(\tilde{G}_i) \cap \tilde{G}_i^{\prime}$. In particular, since by \eqref{delta map}, $\delta(Z(\tilde{G}_1)\cap \tilde{G}_1')= Z(\tilde{G}_2)\cap \tilde{G}_2'$, we can take $A_2=\delta(A_1).$	
 	So the map $\tilde{\delta}:\mathrm{Hom}(A_1, \bC^\times) \to \mathrm{Hom}(A_2, \bC^\times)$ defined
 	by $\chi \to \chi \circ \delta^{-1}|_{A_2}$ is an isomorphism. Since $|\frac{Z(\tilde{G}_1)}{A_1}|=|\frac{Z(\tilde{G}_2)}{\delta(A_1)}|$, we can define a bijection $\eta:\mathrm{Hom}(Z(\tilde{G}_1), \bC^\times) \to \mathrm{Hom}(Z(\tilde{G}_2), \bC^\times)  $ such that $\eta(\chi')|_{A_2}=\tilde{\delta}(\chi'|_{A_1})$, for $\chi' \in \mathrm{Hom}(Z(\tilde{G}_1), \bC^\times)$
 	
 	
 	
 	
In view of Theorem \ref{lifting}, it is enough to prove that, for each $\chi \in \mathrm{Hom}(A_1, \bC^\times)$,
 	the following sets are in dimension preserving bijection 
 	\[\mathrm{Irr}(\tilde{G_1}|\chi) \longleftrightarrow \mathrm{Irr}(\tilde{G_2}|\tilde{\delta}(\chi)).\] 	
 	Now, we have	
 	\[\mathrm{Irr}(\tilde{G}_1|\chi)= \underset{\{\chi' \in \mathrm{Hom}(Z(\tilde{G}_1), \bC^\times) ~: ~\chi'|_{A_1}=\chi\}}\bigcup \mathrm{Irr}(\tilde{G}_1|\chi'),\]
 	\[\mathrm{Irr}(\tilde{G}_2|\tilde{\delta}(\chi))= \underset{\{\chi' \in \mathrm{Hom}(Z(\tilde{G}_1), \bC^\times)  ~: ~ \eta(\chi')|_{A_2}=\tilde{\delta}(\chi)\}}\bigcup \mathrm{Irr}(\tilde{G}_2|\eta(\chi')).\]
 	 	By Theorem \ref{representation tra}$(ii)$, there is a bijective correspondence between  the sets $\mathrm{Irr}(\tilde{G}_1|\chi')$ and $\mathrm{Irr}(\tilde{G}_2|\eta(\chi'))$. Moreover,  $$|\{\chi' \in \mathrm{Hom}(Z(\tilde{G}_1), \bC^\times)  ~: ~ \chi' |_{A_1}=\chi\}|=|\{\chi' \in \mathrm{Hom}(Z(\tilde{G}_1), \bC^\times)  ~: ~ \eta(\chi') |_{A_2}=\tilde{\delta}(\chi)\}|.$$ 
 	Hence, we get the required bijection $\mathrm{Irr}(\tilde{G}_1|\chi) \longleftrightarrow \mathrm{Irr}(\tilde{G}_2|\tilde{\delta}(\chi)).$
	
	\end{proof}
\begin{remark} \label{non-isoclinic rep-group}
	If $G_1$ and $G_2$ are two isoclinic groups of same order with isomorphic Schur multiplier, whose representation groups are not isoclinic, then $G_1$ and $G_2$ may or may not be related.
	\begin{enumerate}[(i)]
		\item It is easy to check that the groups $\Phi_{11}(21^4)a$ and $\Phi_6(21^4)d$ are not isoclinic, and they are the representation groups of $\Phi_2(211)a$ and $\Phi_2(211)c$ respectively. From \cite[Theorem 4.3]{Margolis}, $\Phi_2(211)a \sim_\bC \Phi_2(211)c$.
		\item The groups $\Phi_{23}(21^4)d$ and $\Phi_{(43,2r)}$(see \cite{p^6presentations} for notation) are not isoclinic, and they are the representation groups of $\Phi_6(2111)a$ and $\Phi_6(221)d_0$ respectively. But by \cite[Theorem 4.3]{TGRIP}, $\Phi_6(2111)a \nsim_\bC \Phi_6(221)d_0$.
	\end{enumerate}
\end{remark} 

 \textbf{Proof of Corollary  \ref{class 2}.}
 \begin{proof}
 Suppose $G_1$ and $G_2$ are isoclinic  via the isomorphisms $(\phi,\delta)$ and $G_1$ and $G_2$ are minimally generated by some sets $X$ and $Y$ respectively. Our first claim is that their is a bijection $\Phi:X \to Y$ such that the sets $W_{G_1}^X\longleftrightarrow W_{G_2}^{\Phi(X)}$ are in bijective correspondence.
 	Let $G_1/Z(G_1)=\langle\bar{x}_1,\bar{x}_2,\ldots,\bar{x}_m \rangle$ and $G_2/Z(G_2)=\langle\bar{x}'_1,\bar{x}'_2,\ldots,\bar{x}'_m \rangle$, where $\phi(\bar{x}_i)=\bar{x}'_i$. 
 	If $d(G_1/Z(G_1)) = d(G_1)$,  taking $X=\{x_1,x_2, \ldots,x_m\}$ and $Y=\{x_1',x_2', \ldots,x_m'\}$
 	we define a map $\Phi: X\to Y$ by $ \Phi(x_i)=x_i'$, then $W_{G_1}^X\longleftrightarrow W_{G_2}^{\Phi(X)}$. 
 	Otherwise, if $d(G_1/Z(G_1)) < d(G_1)$, then there is a minimal generating set
 	$X=\{x_1,x_2, \ldots,x_m,
	y_1,y_2,\ldots,y_r\}$ of $G$ such that $x_i \notin Z(G_1)$ for  $i=1,2,\ldots,m$ and $y_j \in Z(G_1)$ for  $j=1,2,\ldots,r.$  Since $d(G_1)=d(G_2)$,  $G_2$ is generated by a set $Y=\langle x'_1,x'_2, \ldots,x'_m,
	y'_1, y'_2,\ldots,y'_r \rangle$ such that $y_j' \in Z(G_2)$. 
 	By defining a map $\Phi: X\to Y$ by $ \Phi(x_i)=x_i', \Phi(y_j)=y_j'$, 
 	our claim is proved. \\

 	Now our aim is to construct a representation group  $\tilde{G}_i$ of $G_i$ ($i=1,2$)  such that  $\tilde{G}_1$ is isoclinic to $\tilde{G}_2$. 
 	Then the result will follow from Theorem \ref{representation related}.
 	We will do it by a finite number of steps. In the first step, we will construct a group $\tilde{G}_1^{(1)}$ which fits in the following exact sequence 
 	\[\xymatrix
 	{1 \ar[r] & \langle z_1 \rangle \ar[r] & \tilde{G}_1^{(1)} \ar[r] & G_1 \ar[r] & 1,}\]  
 	for some $\langle z_1 \rangle \subseteq \tilde{G}_1^{(1)'} \cap Z(\tilde{G}_1^{(1)}) $ of order $p$.
 	Since $W_{G_1}^X$ contains three types of elements $ (x_i,y_j), (y_i,y_j), (x_i,x_j)$,
 	we consider the following three cases here.\\

 \noindent	\textbf{Case(i):} When $(x_i,y_j) \in W_{G_1}^X,$ for some $i$ and $j$. Fix one such pair $(i,j)$.
 	Let $G_1$ have a presentation of the form $\langle X \mid R_1 \rangle$. 
	We define a group of nilpotency class 2, say $\tilde{G}_1^{(1)}=\langle X \mid R_1^{(1)} \rangle$, where
\[
R_1^{(1)}
=
\bigl(R_1\setminus\{z_1\}\bigr)
\cup
\{z_1^p\}
\cup
\{[z_1,x] \mid x\in X\}, \text{ taking } z_1=[x_i,y_j].
\]	
Note that $\tilde{G}_1^{(1)}/Z(\tilde{G}_1^{(1)}) \cong G_1/Z(G_1) \times \langle \bar{y}_j \rangle\cong G_1/Z(G_1) \times \bZ/p\bZ $.
 	Next, we construct a group $\tilde{G}_2^{(1)}$ such that
 	 the following sequence is exact.
 	\[\xymatrix
 	{1 \ar[r] & \langle z_2 \rangle \ar[r] & \tilde{G}_2^{(1)} \ar[r] & G_2 \ar[r] & 1},\]   for some $\langle z_2 \rangle \subseteq \tilde{G}_2^{(1)'} \cap Z(\tilde{G}_2^{(1)}) $ of order $p$, such that  
$\tilde{G}_2^{(1)}$ is isoclinic to $\tilde{G}_1^{(1)}$. 	
 	Observe that $(x_i',y_j') \in W_{G_2}^{\Phi(X)}$, and take $z_2=[x_i',y_j']$. 
 	Let $G_2=\langle Y \mid R_2 \rangle$. We define $\tilde{G}_2^{(1)}=\langle Y \mid R_2^{(1)} \rangle$, where
 	$$R_2^{(1)}=(R_2 \setminus \{z_2\}) \cup \{z_2^p\} \cup \{[z_2,y] \mid y \in Y\}.$$
 	We have $\tilde{G}_2^{(1)}/Z(\tilde{G}_2^{(1)}) =\langle \bar{x}'_1,\bar{x}'_2,\ldots,\bar{x}'_m\rangle \times \langle \bar{y}'_j \rangle\cong G_2/Z(G_2) \times \bZ/p\bZ $.	
 	As $\tilde{G}_i^{(1)}/Z(\tilde{G}_i^{(1)}) \cong G_i/Z(G_i) \times \bZ/p\bZ$, we define a map $$\phi^{(1)}:\tilde{G}_1^{(1)}/Z(\tilde{G}_1^{(1)}) \to \tilde{G}_2^{(1)}/Z(\tilde{G}_2^{(1)})$$ on the generators by 
 	$\phi^{(1)}(\bar{x_i})=\bar{x}'_i$ and $\phi^{(1)}(\bar{y}_j)=\bar{y}'_j$. It is easy to see that it induces an isomorphism as $\phi$ is an isomorphism.
 	Since $\tilde{G}_i^{(1)'}\cong G_i' \times  \bZ/p\bZ$, the map
 	$\delta^{(1)}:\tilde{G}_1^{(1)'} \to \tilde{G}_2^{(1)'}$ defined 
 	by $\delta^{(1)}([x_i,y_j])=[x'_i,y'_j]$ and
 	$\delta^{(1)}([x_t,x_s])=[x'_t,x'_s]$ for $t,s \in \{1,2,\cdots,m\}$, induces an isomorphism.
 	Therefore, $\tilde{G}_1^{(1)}$ and $\tilde{G}_2^{(1)}$ are isoclinic via the isomorphisms $(\phi^{(1)},\delta^{(1)})$.\\

 \noindent	\textbf{Case(ii):} Suppose $(y_i,y_j) \in W_{G_1}^X,$ for some $i$ and $j$. Fix one such pair $(i,j)$. Taking $z_1=[y_i,y_j]$, analogously, as discussed in case(i), we can construct the groups $\tilde{G}_1^{(1)}$ and $\tilde{G}_2^{(1)}$ such that $\tilde{G}_1^{(1)}$ and $\tilde{G}_2^{(1)}$ are isoclinic.\\
 	
 \noindent	\textbf{Case(iii):} If $(x_i,x_j) \in W_{G_1}^X,$ for some $i$ and $j$,. Fix one such pair $(i,j)$. Then taking $z_1=[x_i,x_j]$, similar way we can construct  $\tilde{G}_1^{(1)}$ and $\tilde{G}_2^{(1)}$ such that $\tilde{G}_1^{(1)}$ and $\tilde{G}_2^{(1)}$ are isoclinic.\\

If $k=1$ we are done, otherwise we proceed in the second step. Using the same process discussed above, we can construct the groups $\tilde{G}_i^{(2)}$ such that 
for some $\langle z_i' \rangle \subseteq \tilde{G}_i^{(2)'} \cap Z(\tilde{G}_i^{(2)})$ of order $p$, the sequences
 	\[\xymatrix
 	{1 \ar[r] & \langle z_i' \rangle \ar[r] & \tilde{G}_i^{(2)} \ar[r] & \tilde{G}_i^{(1)} \ar[r] & 1}\text{ and } 
	\xymatrix
 	{1 \ar[r] & \langle z_i,z_i' \rangle \ar[r] & \tilde{G}_i^{(2)} \ar[r] & G_i  \ar[r] & 1},\] 
 are exact, and $\tilde{G}_1^{(2)}$ is isoclinic to $\tilde{G}_2^{(2)}$.

Since $
|W_{G_1}^{X}| \ge k,
$
 after $k$ many steps we have the groups $\tilde{G}_1(=\tilde{G}_1^{(k)})$ and $\tilde{G}_2(=\tilde{G}_2^{(k)})$ such that $\tilde{G}_1$ and $\tilde{G}_2$ are isoclinic and for $A_i \cong   (\bZ/p\bZ)^k \subseteq \tilde{G_i}' \cap Z(\tilde{G_i})$, the sequences 
$\xymatrix
 	{1 \ar[r] & A_i \ar[r] & \tilde{G_i} \ar[r] & G_i \ar[r] & 1}$ are exact for $i=1,2$. 
 	Hence, the proof follows.
\end{proof} 
 
 \subsection{Applications}\label{Application_1}
In this section, we apply Theorem~\ref{TGRIP isoclinic},Theorem~\ref{representation related} and Corollary~\ref{class 2} to determine the $\mathbb{C}$-twist isomorphism classes of several families of finite groups, thereby obtaining solutions to (TGRIP) for these groups.

Throughout this section, we use the notation and terminology of \cite[Lemmas~2.1--2.3]{Rocco1991} and \cite[Lemma~9]{Blyth2009} in the computation of Schur multipliers.
 \begin{thm}\label{Schur Phi_4}
 	We have the following $\sim_\bC$-twist isomorphism classes.
\begin{enumerate}[(i)]
\item $\Phi_4(2211)a
\sim_{\mathbb C}
\Phi_4(2211)c
\sim_{\mathbb C}
\Phi_4(2211)d_r\ (r\neq\tfrac{1}{2}(p-1))
\sim_{\mathbb C}
\Phi_4(2211)e
\sim_{\mathbb C}
\Phi_4(2211)f_r \\
(r\neq 0)
\sim_{\mathbb C}
\Phi_4(2211)g
\sim_{\mathbb C}
\Phi_4(2211)h
\sim_{\mathbb C}
\Phi_4(2211)i$;

\item $\Phi_4(321)a
\sim_{\mathbb C}
\Phi_4(321)c$;

\item $\Phi_4(3111)a
\sim_{\mathbb C}
\Phi_4(3111)b
\sim_{\mathbb C}
\Phi_4(3111)c$;

\item $\Phi_4(21^4)a
\sim_{\mathbb C}
\Phi_4(21^4)b
\sim_{\mathbb C}
\Phi_4(21^4)c$.
\end{enumerate} 
\end{thm}
\begin{proof}
We first determine the Schur multipliers of the groups listed in the statement.
For the groups that are direct products, the required Schur multipliers follow from Theorem~\ref{Schur direct product} together with \cite[Tables 1 and 2]{hatuip5}. For other groups $G$, we follow the arguments given in the proof of \cite[Lemma 5.7]{hatuip5} for computing the Schur multiplier of the group $\Phi_3(2111)d$.

Let $G=\Phi_4(2211)g$. By \cite[Proposition 2.11]{hatuip5}, the non-abelian exterior square $G\wedge G$ is generated by
\[
\left\{
\begin{array}{cl}
&[\alpha_1,\alpha_2^\phi],\;[\beta_1,\beta_2^\phi],\;[\alpha,\gamma^\phi],[\alpha_i,\beta_j^\phi], \text{ for } i,j\in\{1,2\},[\alpha_k,\alpha^\phi],\;[\alpha_k,\gamma^\phi], \text{ for } k\in\{1,2\},\\
&[\beta_m,\alpha^\phi],\;[\beta_m,\gamma^\phi], \text{ for } m\in\{1,2\}
\end{array}
\right\}.
\]
Using \cite[Lemma 2.9]{hatuip5}, the following relations hold in $G\wedge G$. For $x\in\{\alpha,\alpha_1,\alpha_2\}$,
\begin{align*}
[\beta_1,x^\phi]
&=[\alpha^p,x^\phi]
=[\alpha,x^\phi]^p,
&
[\beta_2,x^\phi]
&=[\gamma^p,x^\phi]
=[\gamma,x^\phi]^p,
\\
[\beta_1,\gamma^\phi]
&=[\alpha_1,\alpha,\gamma^\phi]
=1
=[\alpha_2,\alpha,\gamma^\phi]
=[\beta_2,\gamma^\phi],
&
[\beta_1,\beta_2^\phi]
&=[\alpha_1,\alpha,\beta_2^\phi]
=1,
\\
[\alpha_1,\alpha^\phi]^p
&=[\alpha_1^p,\alpha^\phi]
=1
=[\alpha_2^p,\alpha^\phi]
=[\alpha_2,\alpha^\phi]^p,
&
[\alpha_1,\alpha_2^\phi]^p
&=[\alpha_1^p,\alpha_2^\phi]
=1,
\\
[\alpha_1,\gamma^\phi]^p
&=[\alpha_1^p,\gamma^\phi]
=1
=[\alpha_2^p,\gamma^\phi]
=[\alpha_2,\gamma^\phi]^p,
&
[\alpha,\gamma^\phi]^p
&=[\alpha^p,\gamma^\phi]
=[\beta_1,\gamma^\phi]
=1.
\end{align*}

Therefore,
$
G\wedge G=
\langle
[\alpha_1,\alpha^\phi],
[\alpha_2,\alpha^\phi],
[\alpha_1,\alpha_2^\phi],
[\alpha_1,\gamma^\phi],
[\alpha_2,\gamma^\phi],
[\alpha,\gamma^\phi]
\rangle.
$
Next consider the quotient group
$A=G/\langle\beta_1\rangle\cong\Phi_2(2111)b$.
By \cite[Table 2]{hatuip5},
$\M(A)\cong(\bZ/p\bZ)^5$.
Hence, by \cite[Theorem~3.1]{jones1973},
$|\M(G)|\ge p^4$, which implies $|G\wedge G|\ge p^6$.
Therefore, the above generating set is minimal and $
G\wedge G\cong(\bZ/p\bZ)^6$ by \cite[lemma 2.12]{hatuip5}.
Consequently,
$$
\M(G)=\langle
[\alpha_1,\alpha_2^\phi],
[\alpha_1,\gamma^\phi],
[\alpha_2,\gamma^\phi],
[\alpha,\gamma^\phi]
\rangle
\cong(\bZ/p\bZ)^4.
$$

The remaining Schur multipliers are computed by analogous arguments and are given as follows.
\begin{enumerate}[(a)]
\item $\M(G)\cong(\bZ/p\bZ)^4$ if $G$ is isomorphic to any of the groups $\Phi_4(2211)a$, $\Phi_4(2211)c$, $\Phi_4(2211)d_r$ $(r\neq\tfrac{1}{2}(p-1))$, $\Phi_4(2211)e$, $\Phi_4(2211)f_r~(r \neq 0)$, $\Phi_4(2211)g$, $\Phi_4(2211)h$, or $\Phi_4(2211)i$.

\item $\M(G)\cong\bZ/p\bZ\times\bZ/p\bZ$ if $G$ is isomorphic to either $\Phi_4(321)a$ or $\Phi_4(321)c$.

\item $\M(G)\cong(\bZ/p\bZ)^3$ if $G$ is isomorphic to any of the groups $\Phi_4(3111)a$, $\Phi_4(3111)b$, or $\Phi_4(3111)c$.

\item $\M(G)\cong(\bZ/p\bZ)^6$ if $G$ is isomorphic to any of the groups $\Phi_4(21^4)a$, $\Phi_4(21^4)b$, or $\Phi_4(21^4)c$.
\end{enumerate}

We now determine the $\bC$-twist isomorphism classes.

\smallskip

\noindent
(i) Since all the groups listed in (i) are nilpotent of class $2$, the result follows immediately from Corollary~\ref{class 2}.

\smallskip

\noindent
(ii) Let
$G_1=\Phi_4(3111)a$,
$G_2=\Phi_4(3111)b$, and
$G_3=\Phi_4(3111)c$.
Consider the groups
\begin{align*}
	\tilde{G}_i  = &
	\langle
	\alpha,\alpha_1,\alpha_2,\beta_j, ~ 1\le j\le5
	\mid
	[\alpha_1,\alpha]=\beta_1,
	[\alpha_2,\alpha]=\beta_2,
	[\alpha_1,\alpha_2]=\beta_3,
	[\beta_2,\alpha_2]=\beta_4,\\
	& 
	[\beta_2,\alpha]=\beta_5,
	R_i
	\rangle,
\end{align*}
where
\begin{align*}
R_1:\;&
\alpha^{p^2}=\beta_1,~~~~~
\alpha_1^p=\alpha_2^p=\beta_j^p=1,~~ 1\le j\le5,\\
R_2:\;&
\alpha_1^{p^2}=\beta_1,~~~~
\alpha^p=\alpha_2^p=\beta_j^p=1,~~ 1\le j\le5,\\
R_3:\;&
\alpha_2^{p^2}=\beta_1,~~~
\alpha^p=\alpha_1^p=\beta_j^p=1,~~1\le j\le5.
\end{align*}
 Note that
 $$
\tilde{G}_1
\cong
\left(
\langle
\alpha,\alpha_1,\beta_1,\beta_2,\beta_5 \rangle \times\langle\beta_3\rangle\times\langle\beta_4\rangle
\right)\rtimes\langle\alpha_2\rangle
, ~
\tilde{G}_2
\cong
\left(
\langle
\alpha,\alpha_1,\beta_1,\beta_2,\beta_5\rangle\times\langle\beta_3\rangle\times\langle\beta_4\rangle
\right)\rtimes\langle\alpha_2\rangle,
$$ and $
\tilde{G}_3
\cong
\left(
\langle
\alpha,\alpha_2,\beta_2,\beta_4,\beta_5\rangle \times\langle\beta_3\rangle
\right)\rtimes\langle\alpha_1\rangle.
$
Since 
$\langle\beta_3,\beta_4,\beta_5\rangle
\subseteq
Z(\tilde{G}_i)\cap\tilde{G}_i'$
such that 
$\tilde{G}_i/
\langle\beta_3,\beta_4,\beta_5\rangle
\cong
G_i$, each $\tilde{G}_i$ is a representation group of $G_i$.
Moreover, it is easy to check that
$\tilde{G}_i$
are isoclinic to each other.
Hence, by Theorem~\ref{representation related},
$
G_1
\sim_{\bC}
G_2
\sim_{\bC}
G_3.
$

\smallskip

\noindent
(iii) The proof is analogous to that of (ii).

\smallskip

\noindent
(iv) Let $G_1=\Phi_4(21^4)a=\langle \alpha,\alpha_i,\beta_i,i=1,2 \mid [\alpha_i,\alpha]=\beta_i,\alpha^p=\beta_2,\alpha_i^p=\beta_i^p=1 \times \langle \alpha_3 \rangle$
and $$G_2=\Phi_4(21^4)b=\langle \alpha',\alpha_i',\beta_i',i=1,2 \mid [\alpha_i',\alpha']=\beta_i',\alpha_1'^p=\beta_1',\alpha'^p=\alpha_2'^p=\beta_i'^p=1 \times \langle \alpha_3' \rangle.$$
 Then one can easily check that $Z^*(G_i)=G_i^p$. Furthermore, $G_1$ and $G_2$ are isoclinic via the central subgroups $(G_1^p,G_2^p)$ and via the isomorphism $(\phi,\delta)$, where $\phi:G_1/G_1^p \to G_2/G_2^p$ is defined on the generators by $\phi(\alpha)=\alpha',\phi(\alpha_1)=\alpha_2'$ and $\phi(\alpha_2)=\alpha_1'$ and, $\delta:G_1' \to G_2'$ is defined by $\delta(\beta_1)=\beta_2'$ and $\delta(\beta_2)=\beta_1'$. Hence by Theorem~\ref{TGRIP isoclinic}, $G_1\sim_{\bC}G_2$. A similar argument shows that $\Phi_4(21^4)b \sim_{\bC} \Phi_4(21^4)c$. 

This completes the proof.
\end{proof}

In the next result, we present another family of groups. The Schur multipliers of the groups listed below are obtained either from Theorem~\ref{Schur direct product} together with \cite[Tables 1 and 2]{hatuip5}, or by arguments analogous to those used in the proof of Theorem~\ref{Schur Phi_4} for the group $\Phi_4(2211)g$. The desired $\mathbb{C}$-twist isomorphism classes then follow immediately from Corollary~\ref{class 2} by comparing their Schur multipliers.
  \begin{thm}
 	We have the following $\sim_\bC$-twist isomorphism classes.
 	\begin{enumerate} [(i)]
 		\item $\Phi_2(411)a\sim_{\mathbb C}\Phi_2(321)e \sim_\bC\Phi_2(411)b;$
 		\item $\Phi_2(3111)a \sim_\bC \Phi_2(3111)b \sim_\bC \Phi_2(2211)b \sim_\bC \Phi_2(2211)e;$
 		\item $\Phi_2(21^4)a \sim_\bC \Phi_2(21^4)b$.
 	\end{enumerate}
	 \end{thm}

 \section{Special $p$-groups} \label{Special p-groups}	
 This section deals with the isoclinic non-capable special $p$-groups.
 We present a   proof of Theorem \ref{TGRIP special p group} here.
Our first result provides a solution of (TGRIP) for isoclinic unicentral special $p$-groups.
 \begin{thm} \label{TGRIP unicentral}
 	Let $G_1$ and $G_2$ be two isoclinic special $p$-groups which are unicentral. Then $G_1 \sim_\mathbb{C} G_2.$ 
 \end{thm}
 
 \begin{proof}
 		Proof follows from Corollary \ref{isoclinic1}.
 \end{proof}
 However the following example says that two non-isoclinic special $p$-groups of same rank which are unicentral, may not be $\mathbb C$-twist isomorphic.
  \begin{example}
Let $G_1=\Phi_{12}(2211)a$ and $G_2=\Phi_{13}(2211)a$. Then $G_1\nsim_\mathbb C G_2$ since $\bC G_1\ncong \bC G_2$.
 \end{example} 

 \begin{thm} \label{TGRIP rank 2}
 	Let $G_1$ and $G_2$ be two non-capable special p-groups of rank $2$ such that the following properties hold:
 	\begin{enumerate}[(i)]
 		\item $\M(G_1) \cong \M(G_2)$;
 		\item $G_1$ and $G_2$ are isoclinic.
 	\end{enumerate} 	
 	Then, $G_1 \sim_\mathbb{C} G_2.$
 \end{thm}
 
 \begin{proof}
Since $G_1$ and $G_2$ are isoclinic, $\bC G_1 \cong \bC G_2$. Note that $|G_1|=|G_2|$, and it follows from \cite{schurrank2} that  $|\M(G_i)|=p^{\frac{1}{2}d(d-1)-t}$ with $t=0$ or $2$, where $d=d(G_i)$. 

If $|\M(G_i)|=p^{\frac{1}{2}d(d-1)-2}$ for $i=1,2$, then $Z^*(G_i)=Z(G_i)$, by \cite{schurrank2}. Hence,  by Theorem $\ref{TGRIP unicentral}$, $G_1\sim_{\bC} G_2$.

 Now, let $|\M(G_i)|=p^{\frac{1}{2}d(d-1)}$, then by \cite{schurrank2} $Z^*(G_i) \cong \bZ/p\bZ$. Suppose $G_1=\langle x_1,x_2,\ldots,x_d \rangle$ and $G_2=\langle x_1',x_2',\ldots ,x_d'\rangle$ 
 We consider the following cases by comparing $G_1^p$ and $G_2^p$. \\
	 
 \textbf{Case (a):} First assume that $ Z^*(G_i)=G_i^p \cong \bZ/p\bZ$. Then it follows from \cite[Theorem 1.3(d)]{schurrank2},	$G_i/{G_i}^p \cong \mathrm{ES}_p(p^{3})\times (\bZ/p\bZ)^{d-2}$.	
Let
$$G_1/G_1^p=\langle \bar{x}_1,\bar{x}_2 \mid [\bar{x}_1,\bar{x}_2]=\bar{z}_1 \rangle \times \langle \bar{x}_3 \rangle \times \cdots \times \langle \bar{x}_d \rangle.$$
Then $G_1$ has the following type of presentation,
$$G_1=\langle x_1,x_2,\ldots,x_d \mid [x_1,x_2]=z_1, [x_i,x_j]=z_2^{i * j},1 \le i <j \le d, z_2 \in G_1^p  \rangle,$$
for $i * j \in \{0,1,2,\cdots, p-1\}$ and $(i,j)\neq (1,2). $ 
Now since $G_1$ and $G_2$ are isoclinic, $G_2$ has the following type presentation,
$$G_2=\langle x_1',x_2',\ldots,x_d' \mid [x_1',x_2']=z_1', [x_i',x_j']=z_2'^{i * j},1 \le i <j \le d \rangle. $$
We first assume that  there are $s,t \notin \{1,2\}$ such that $[x_s',x_t']=z_2'^{s * t}\neq 1$. If possible, let $G_2^p= \langle {z_1'}^m{z_2'}^n \rangle$ for  some $m,n \in \{1,2,\cdots, p-1\}$. Then
$$G_2/G_2^p = \langle  \bar{x}_1',\bar{x}_2', \ldots, \bar{x}_d' \mid [\bar{x}_1',\bar{x}_2']=\bar{z}_1', [\bar{x}_s',\bar{x}_t']=\bar{z}_1'^{-(s * t)mn^{-1}}, [\bar{x}_i',\bar{x}_j']=\bar{z}_1'^{-(i*j)mn^{-1}} \rangle. $$
But this contradicts the fact that $G_2/G_2^p \cong \mathrm{ES}_p(p^{3})\times (\bZ/p\bZ)^{d-2}$. Hence, $G_2^p=\langle z_1' \rangle$ or $G_2^p=\langle z_2' \rangle$. 
Suppose $G_2^p=\langle z_1' \rangle$. 
Since $G_2/G_2^p \cong \mathrm{ES}_p(p^{3})\times (\bZ/p\bZ)^{d-2}$, there can not be  distinct $s,t, u,v \in \{1, 2, \cdots, d\}$ such that $[\bar x_s',\bar x_t']$ and $[\bar x_u', \bar x_v']$ are non-trivial in $G_2/G_2^p$.
Therefore, for $d \geq 5$, $G_2$ will be of following types:
\begin{align*}
	G_2  = & \langle {x}_1',{x}_2', x_s',x_t',x_u' \mid [{x}_1',{x}_2']={z}_1', [{x}_s',{x}_t']={z}_2'^{s * t}, [{x}_s',{x}_u']={z}_2'^{s*u},[{x}_t',{x}_u']={z}_2'^{t*u} \rangle\\
	\cong & \langle {x}_1',~{x}_2',~{x}_s',~ {x}_t',~ {x}_s'^{-(t*u)(s * t)^{-1}}{x}_u'^{-1}x_t'^{(s*u)(s * t)^{-1}} \mid [{x}_1',{x}_2']={z}_1',[{x}_s',{x}_t']={z}_2'^{s * t} \rangle,
\end{align*}
\begin{align*}
	\text{or,} ~	G_2  = & \langle {x}_1',{x}_2', \ldots, {x}_d' \mid [{x}_1',{x}_2']={z}_1', [{x}_s',{x}_t']={z}_2'^{s * t}, [{x}_s',{x}_j']={z}_2'^{s*j}, 1 \le j \le d, j \neq t \rangle\\
	\cong & \langle {x}_s',~{x}_t',~ {x}_t'^{-(s*j)(s * t)^{-1}}{x}_j', 1 \le j \le d, j \neq t \mid [{x}_s',{x}_t']={z}_2'^{s * t},\\
	&	[{x}_t'^{-(s*1)(s * t)^{-1}}{x}_1',{x}_t'^{-(s*2)(s * t)^{-1}}{x}_2']=z_1' \rangle,
\end{align*}
and these groups are not special $p$-groups of rank 2. 
For $d= 4$, $|G_2|=p^6$, so by \cite{james} it follows that,
$G_2$ is isomorphic to $\Phi_{12}(21^4)a$ or $\Phi_{12}(21^4)d$. In this case, $G_1$ is isomorphic to $\Phi_{12}(21^4)c$. Then, it is easy to check that $G_1$ and $G_2$ are isoclinic via the central subgroups $(G_1^p,G_2^p)$.  
Now taking $A_i=G_i^p$ in Theorem \ref{TGRIP isoclinic}, $G_1 \sim_\bC G_2$.

Next, let there be no such non-trivial commutator $[x_i',x_j']$ in the presentation of $G_2$ such that $i,j \notin \{1,2\}$, then $G_2$ is of the following type: 
$$G_2=\langle x_1',x_2',\ldots,x_d' \mid [x_1',x_2']=z_1', [x_1',x_j']=z_2'^{1 * j},[x_2',x_j']=z_2'^{2 * j},3 \le j \le d \rangle. $$

Again if $G_2^p= \langle {z_1'}^m{z_2'}^n \rangle$ for  some $m,n \in \{1,2,\cdots, p-1\}$ or $G_2^p=\langle z_1' \rangle$, then since $G_2/G_2^p \cong \mathrm{ES}_p(p^{3})\times (\bZ/p\bZ)^{d-2}$, $G_2$ are of the following types:
\begin{align*}
	G_2 & =\langle x_1',x_2',x_3' \mid [x_1',x_2']=z_1', [x_2',x_3']=z_2'^{l_1}, [x_1',x_3']=z_2'^{l_2}, z_2' \in G_2^p  \rangle \\
	& \cong \langle x_2',x_3',x_2^{-l_2l_1^{-1}}x_1 \mid [x_2',x_3']=z_2'^{l_2},[x_2',x_2'^{-l_2l_1^{-1}}x_1']=z_1'^{-1}, z_2' \in G_2^p \rangle,
\end{align*}
\begin{align*}
	\text{or,}\quad G_2 & = \langle x_1',x_2',\ldots,x_d' \mid [x_1',x_2']=z_1', [x_1',x_j']=z_2'^{l_j},3 \le j \le d, z_2' \in G_2^p  \rangle\\
	& \cong \langle x_1',x_2',x_3', x_3'^{-l_jl_3^{-1}}x_j', 4 \le j\le d \mid [x_1',x_2']=z_1',[x_1',x_3']=z_2'^{l_3},z_2' \in G_2^p \rangle.
\end{align*}
For $d\geq 4$, $G_2$ are not a special $p$-group of rank $2$. For $d=3$, $|G_i|=p^5$ and by \cite{james}, each $G_i$ is isomorphic to one of the following groups.
$$\{\Phi_4(2111)a,\Phi_4(2111)b,\Phi_4(2111)c\}.$$
Then they are in the same $\bC$-twist isomorphic class, follows from \cite[Theorem 4.3]{TGRIP}.

Finally, if $G_2^p=\langle z_2' \rangle$, then
$G_2/G_2^p=\langle \bar{x}_1',\bar{x}_2' \mid [\bar{x}_1',\bar{x}_2']=\bar{z}_1' \rangle \times \langle \bar{x}_3' \rangle \times \cdots \times \langle \bar{x}_d' \rangle.$
Thus, it is easy to see that $G_1$ and $G_2$ are isoclinic via the central subgroups $(G_1^p,G_2^p)$.  
Now taking $A_i=G_i^p$ in Theorem \ref{TGRIP isoclinic}, $G_1 \sim_\bC G_2$.
	
 	\textbf{Case (b):} Next consider $G_i^p=1$. Since up to isomorphism there is only one special $p$-group of exponent $p$ in a fixed isoclinic family, we are done.
 	 	
 	\textbf{Case (c):} 
 Now we assume that $G_1^p =Z^*(G_1)$ and $G_2^p=1$.  
 By \cite[Theorem 5.1(iv),(v)]{mazur}, $G_2$ will have the following presentations.
 
 (i) Let $d=2m+1$ where $m \ge 2$. Then 
 \begin{equation} \label{presentation of G_2}
 	G_2=\langle x_i',z_1',z_2', 1\leq i \leq 2m+1\mid x_i'^p=z_j'^p=1,[x_1',x_{2m+1}']=z_1',[x_k',x_{m+k}']=z_2', 1\leq k \leq m \rangle
 \end{equation}
 with $Z^*(G_2)=\langle z_2' \rangle.$
 
 (ii) let $d=2m$, where $m \ge 3$. Then
 $$G_2= \langle x_i',z_1',z_2', 1\leq i \leq 2m \mid x_i'^p=z_j'^p=1, [x_1',x_{m+1}']=z'_1,[x_k',x'_{m+k}]=z'_2, 2 \le k \le m \rangle $$
 or
 $$G_2= \langle x_i',z_1',z_2', 1\leq i \leq 2m \mid x_i'^p=z_j'^p=1, [x_1',x_{m+2}']=z'_1,[x_k',x'_{m+k}]=z'_2, 1 \le k \le m \rangle. $$
Consider $d=2m+1$, then $G_2$ is of the form given in \eqref{presentation of G_2}. Now since by \cite[Theorem 1.3(d), 1.4(h)]{schurrank2}, $G_i/Z^*(G_i) \cong \mathrm{ES}_p(p^{3})\times (\bZ/p\bZ)^{d-2}$ and they are isoclinic, we must have 
 	$$G_1=\langle x_i,z_1,z_2, 1\leq i \leq 2m+1\mid [x_1,x_{2m+1}]=z_1,[x_k,x_{m+k}]=z_2 \text{ for } 1 \le k \le m, z_2 \in G_1^p \rangle.$$
 	Therefore, taking $A_i=Z^*(G_i)$ in  Theorem \ref{TGRIP isoclinic}, the result follows. For other possibilities of $G_2$ proof follows similarly. 
	
 \end{proof}

 \textbf{Proof of Theorem \ref{TGRIP special p group}.}
 \begin{proof}
 Proof follows from Theorem \ref{TGRIP unicentral} and Theorem \ref{TGRIP rank 2}.
 \end{proof}
 
\subsection{Solution of (TGRIP) for Special $p$-Groups of Order $p^6$}\label{Application_2}
As an application of Theorem~\ref{TGRIP rank 2}, we solve (TGRIP) for the special $p$-groups of rank $2$ and of order $p^6$. These groups belong to the isoclinism families $\Phi_{12}$, $\Phi_{13}$ and $\Phi_{15}$. The following result also correct the list given in \cite[Theorem 3.2]{TGRIP}.
\begin{thm}
All non-singleton $\mathbb{C}$-twist isomorphism classes of special $p$-groups of rank $2$ and order $p^6$ are the following:
\begin{enumerate}[(i)]
\item $\Phi_{12}(2211)a \sim_{\mathbb C} \Phi_{12}(2211)b
\sim_{\mathbb C} \Phi_{12}(2211)c \sim_{\mathbb C}
\Phi_{12}(2211)e \sim_{\mathbb C} \Phi_{12}(2211)f \sim_{\mathbb C}\\
\Phi_{12}(2211)g
\sim_{\mathbb C}
\Phi_{12}(2211)h
\sim_{\mathbb C}
\Phi_{12}(2211)i
\sim_{\mathbb C}
\Phi_{12}(21^4)b
\sim_{\mathbb C}
\Phi_{12}(21^4)e$;

\item $\Phi_{12}(21^4)a
\sim_{\mathbb C}
\Phi_{12}(21^4)c
\sim_{\mathbb C}
\Phi_{12}(21^4)d$;

\item $\Phi_{13}(2211)a
\sim_{\mathbb C}
\Phi_{13}(2211)b
\sim_{\mathbb C}
\Phi_{13}(2211)c_r
\sim_{\mathbb C}
\Phi_{13}(2211)d
\sim_{\mathbb C}
\Phi_{13}(2211)e_r\ (r\neq1)
\sim_{\mathbb C}
\Phi_{13}(2211)f
\sim_{\mathbb C}
\Phi_{13}(21^4)a
\sim_{\mathbb C}
\Phi_{13}(21^4)d$;

\item $\Phi_{13}(21^4)b
\sim_{\mathbb C}
\Phi_{13}(21^4)c$;

\item $\Phi_{15}(2211)a
\sim_{\mathbb C}
\Phi_{15}(2211)b_{r,s}
\sim_{\mathbb C}
\Phi_{15}(2211)c
\sim_{\mathbb C}
\Phi_{15}(2211)d_r\ (r\neq1) \sim_{\mathbb C} \Phi_{15}(21^4)$.
\end{enumerate}
\end{thm}

\begin{proof}
By \cite[Table~4.1]{james}, the complex group algebras of the special $p$-groups in the isoclinism families $\Phi_{12}$, $\Phi_{13}$ and $\Phi_{15}$ are pairwise non-isomorphic. 
Therefore, each isoclinism class will form different $\bC$-twist classes. 

First, we compute Schur multipliers of the listed groups. 
 Consider $G=\Phi_{12}(2211)a$. Then $X_1$ is generated by the set $$\{-\alpha_2G'\otimes\alpha_1^p, \beta_2G'\otimes\alpha_1^p, \alpha_1G'\otimes\beta_1^p, -\beta_1G'\otimes\beta_1^p\},$$ while $X_2$ is generated by $$\{\alpha_1G'\otimes\alpha_1^p,\beta_1G'\otimes\beta_1^p,\alpha_2G'\otimes\alpha_1^p,\beta_2G'\otimes\alpha_1^p,\alpha_2G'\otimes\beta_1^p,\beta_2G'\otimes\beta_1^p,\alpha_1G'\otimes\beta_1^p+\beta_1G'\otimes\alpha_1^p\}.$$ Note that $\dim X=8$. Hence, by Theorem~\ref{epicentre2}  $Z^*(G)=Z(G)$, and  $\M(G)\cong(\mathbb{Z}/p\mathbb{Z})^4$ by \cite[Theorem~1.1(b)]{schurrank2}.
The remaining groups are handled similarly, using \cite[Theorems~1.1, 1.3 and 1.4]{schurrank2} and  their Schur multipliers are given as follows:
\begin{enumerate}[(a)]
\item $\M(G)\cong(\mathbb{Z}/p\mathbb{Z})^4$ if $G$ is isomorphic to one of $\Phi_{12}(2211)a$, $\Phi_{12}(2211)b$, $\Phi_{12}(2211)c$, $\Phi_{12}(2211)e$, $\Phi_{12}(2211)f$, $\Phi_{12}(2211)g$, $\Phi_{12}(2211)h$, $\Phi_{12}(2211)i$, $\Phi_{12}(21^4)b$, $\Phi_{12}(21^4)e$,\\ $\Phi_{13}(2211)a, \Phi_{13}(2211)b, 
\Phi_{13}(2211)c_r,\Phi_{13}(2211)d , \Phi_{13}(2211)e_r (r \neq 1) , \Phi_{13}(2211)f,$\\ $\Phi_{13}(21^4)a , \Phi_{13}(21^4)d,
\Phi_{15}(2211)a,  \Phi_{15}(2211)b_{r,s} ,\Phi_{15}(2211)c,$  $\Phi_{15}(2211)d_r(r \neq 1)$ or \\ $\Phi_{15}(21^4)$;

\item $\M(G)\cong(\mathbb{Z}/p\mathbb{Z})^6$ if $G$ is isomorphic to $\Phi_{12}(21^4)a$, $\Phi_{12}(21^4)c$, $\Phi_{12}(21^4)d,\Phi_{13}(21^4)b$,\\
or $\Phi_{13}(21^4)c$;

\item $\M(G)\cong\mathbb{Z}/p^2\mathbb{Z}\times(\mathbb{Z}/p\mathbb{Z})^3$ if $G$ is isomorphic to $\Phi_{12}(2211)d,\Phi_{13}(2211)e$, or $\Phi_{15}(2211)d$;

\item $\M(G)\cong(\mathbb{Z}/p\mathbb{Z})^8$ if $G$ is isomorphic to $\Phi_{12}(1^6),\Phi_{13}(1^6)$, or $\Phi_{15}(1^6)$.
\end{enumerate}
 By Theorem~\ref{TGRIP rank 2}, we have the non-singleton $\mathbb{C}$-twist isomorphism classes listed above.
\end{proof} 
 	\section*{Acknowledgements} 
 The authors acknowledge the support of the National Institute of Science Education and Research (NISER), Bhubaneswar and Homi Bhaba National Institute (HBNI), Mumbai.

      \bibliographystyle{amsplain}
  \bibliography{ref1}
 
 \end{document}